\documentclass{article}

\def\DARKMODE{0}

\usepackage{amsmath,amssymb,amsthm}
\usepackage{fullpage}
\usepackage{verbatim}
\usepackage[noadjust]{cite} 
\usepackage{hyperref}
\usepackage[capitalize, nameinlink]{cleveref}
\usepackage{enumitem}
\usepackage[usenames,dvipsnames]{xcolor} 
\usepackage{tikz}
\usepackage[font=small,labelfont=bf, margin=.5in]{caption}

\usepackage[utf8]{inputenc}

\usepackage{parskip, amsfonts, amsmath, changepage, amsthm, setspace, geometry, verbatim, MnSymbol,chemarrow}

\usepackage[T1]{fontenc}

\usepackage{authblk}

\usepackage{xr}

\usepackage{MnSymbol, wasysym, mathtools}

\usepackage[tiny]{titlesec}

\usepackage{hyperref}

\newcommand*\Mod[1]{ \; (\textup{mod} \; #1 )}

\newcommand*{\Z}{\mathbb{Z}}
\newcommand*{\N}{\mathbb{N}}

\renewcommand{\phi}{\varphi}

\DeclarePairedDelimiter{\set}{\{}{\}}

\DeclareMathOperator{\grid}{Gr}

\newtheorem{Thm}{Theorem}[section]

\newtheorem{Lem}[Thm]{Lemma}
\newtheorem{Cor}[Thm]{Corollary}

\theoremstyle{definition}

\theoremstyle{remark}
\newtheorem{Ex}[Thm]{Example}
\newtheorem{Rmk}[Thm]{Remark}

\providecommand{\keywords}[1]
{
  \small	
  \textbf{\textit{Keywords---}} #1
}

\if\DARKMODE1

\else

\fi

\title{Chromatic numbers of circulants with indispensable generators}

\author[a]{Ferdous Ahmed}
\affil[a]{University of California, Irvine, ferdousa@uci.edu}

\author[b]{David Asraf}
\affil[b]{University of California, San Diego, dasraf@ucsd.edu}

\author[c]{David Bonds}
\affil[c]{affiliation and email address not provided}

\author[d]{Jonathan Davidson}
\affil[d]{California State University, Los Angeles, Dept. of Mathematics, 5151 State University Drive, Los Angeles, CA 91711, jdavids6@calstatela.edu}

\author[e]{Yunhee Jang}
\affil[e]{Rice University, Houston, TX, yj78@rice.edu}

\author[f]{Mike Krebs}
\affil[f]{California State University, Los Angeles, Dept. of Mathematics, 5151 State University Drive, Los Angeles, CA 91711, mkrebs@calstatela.edu}

\author[g]{Anand Prakash}
\affil[g]{California State University, Los Angeles, Dept. of Mathematics, 5151 State University Drive, Los Angeles, CA 91711, aprakas3@outlook.com}

\author[h]{Edgar Yak-De Padua}
\affil[h]{California State University, Los Angeles, Dept. of Mathematics, 5151 State University Drive, Los Angeles, CA 91711, eyakdep@calstatela.edu}

\date{\today}

\begin{document}

\maketitle

\keywords{graph, chromatic number, circulant, indispensable, minimal generating set}




\begin{abstract}
The Cayley graph $\text{Cay}(G,S)$ is the graph whose vertex set is the group $G$, where two vertices $x$ and $y$ are adjacent if and only if $xy^{-1}$ or $yx^{-1}$ lies in some fixed subset $S$ of $G$.  We call the elements of $S$ \emph{generators}. A \emph{circulant} graph is a Cayley graph where $G$ is finite and cyclic.  Chromatic numbers of circulant graphs have been studied by many authors.  A general formula due to Heuberger for the chromatic number of a circulant graph is known when $S$ has two elements, but no such formula is known when $S$ has three or more elements. We say that an element $x$ of $S$ is \emph{indispensable} if $S\setminus\{x\}$ does not generate $G$.  We say that $S$ is \emph{minimal} if every element of $S$ is indispensable.  By a result of Garcia-Marco and Knauer from 2024, if $G$ is nilpotent and $S$ is minimal, then $\text{Cay}(G,S)$ is $3$-colorable. In this article, we prove three main results. First, we give an upper bound for the chromatic number of a circulant graph with three generators, one of which is indispensable. Second, we present an alternate proof of the theorem of Garcia-Marco and Knauer for the case of abelian groups. Third, we apply these methods to provide a considerably more systematic (and potentially generalizable) proof of Heuberger's theorem for the chromatic number of circulant graphs with two generators.  Throughout this paper, our primary tool is the theory of Heuberger matrices, for which we provide a brief primer.\end{abstract}

\paragraph*{MSC2020 subject classification: }05C15

\section{Introduction}

\subsection{Overview}


The \emph{chromatic number} of a graph $X$, denoted $\chi(X)$, is the smallest number of colors needed to color the vertices of $X$ so that no two adjacent vertices are assigned the same color.  Given an abelian group $G$ and a subset $S$ of $G$, we say that $S$ is \emph{symmetric} if $-x\in S$ whenever $x\in S$.  (Throughout this paper, we will write all abelian groups using additive notation.)  Elements of $S$ are called \emph{generators}.  (This terminology can be ambiguous, so one must be careful of the context.  Sometimes we begin with a not-necessarily-symmetric set $\tilde{S}$ and then take $S$ to be the set of all elements of $\tilde{S}$ as well as their negatives, yet refer to the elements of $\tilde{S}$ as the generators.)  For such a pair $(G,S)$ we define the \emph{Cayley graph} $\text{Cay}(G,S)$ to be the graph with vertex set $G$ such that $x$ and $y$ are adjacent if and only if $x-y\in S$.  We define a \emph{circulant graph} to be a graph of the form $\text{Cay}(\mathbb{Z}_n,S)$, where $\mathbb{Z}_n$ denotes the group of integers modulo $n$, and $S$ is a symmetric subset of $\mathbb{Z}_n$.  Given integers $n, a_1,\dots,a_k$, we denote the circulant graph $\text{Cay}(\mathbb{Z}_n,\{\pm a_1,\dots,\pm a_k\})$ by $C_n(a_1,\dots,a_k)$.  For example, $C_n(1)$ is an $n$-cycle, and $C_5(1,2)$ is a complete graph on $5$ vertices.  By an abuse of terminology, in this context we often say that $C_n(a_1,\dots,a_k)$ has $k$ generators, namely the elements $a_1,\dots,a_k$.

Chromatic numbers of circulant graphs have been studied by many authors.  In \cite{Heuberger}, Heuberger completely determines the chromatic number of circulant graphs with two generators, i.e., those of the form $C_n(a,b)$.  See Thm. \ref{theorem-Heubergers} below for Heuberger's formula.  A complete solution for the three-generator case is not yet known.  Partial results have been obtained in \cite{Barajas-Serra, Ilic, Meszka, Nicoloso, Nicoloso-solo} and elsewhere.

Let $G$ be an abelian group, written additively, and let $S=\{\pm x_1,\dots,\pm x_m\}$ be a finite symmetric subset of $G$ such that $S$ generates $G$.  We say an $m$-tuple $(a_1,\dots,a_m)^t$ of integers (which we write as a column vector) is a \emph{relation} for $S$ if $a_1 x_1+\cdots +a_m x_m=0$.  The set of relations for $S$ is a subgroup $H$ of $\mathbb{Z}^m$.  We say an $m\times r$ integer matrix $M$ is a \emph{Heuberger matrix} for $(G,S)$ if the columns of $M$ generate $H$.  As discussed in \cite{Cervantes-Krebs}, the matrix $M$ completely encodes all information about the graph $\text{Cay}(G,S)$.  Conversely, given an $m\times r$ integer matrix $M$, we may construct an abelian group $G$ and a symmetric subset $S$ such that $M$ is a Heuberger matrix for $(G,S)$.  To wit, let $H$ be the subgroup of $\mathbb{Z}^m$ generated by the columns of $M$.  Let $G=\mathbb{Z}^m/H$, and let $S=\{H\pm e_1,\dots, H\pm e_m\}$, where $e_i$ is the column vector in $\mathbb{Z}^m$ with a $1$ in the $i$th position and $0$ elsewhere.  In this case we call $\text{Cay}(G,S)$ a \emph{standardized abelian Cayley graph} (or SACG for short), and we denote $\text{Cay}(G,S)$ by $M^{\text{SACG}}$.  If one begins with a pair $(G,S)$ and constructs $M$ as in the preceding paragraph, then indeed $\text{Cay}(G,S)$ is isomorphic to $M^{\text{SACG}}$.  Thus every connected finite-degree abelian Cayley graph is isomorphic to an SACG.

We say that an element $x$ of $S$ is \emph{indispensable} if $S\setminus\{\pm x\}$ does not generate $G$.  We say that $S$ is \emph{minimal} if every element of $S$ is indispensable.  (Observe the slight mismatch between these definitions and those in the abstract; this stems from the ambiguity in the term ``generators.'')

As we prove in Lemma \ref{lem:indispensable-Heuberger}, an indispensable generator corresponds to a non-primitive row (i.e., a row whose elements have $\text{gcd}>1$) of the corresponding Heuberger matrix.  The main idea of this article is to exploit that correspondence, using known facts about how certain matrix operations interact with chromatic numbers.

Our first main theorem (Theorem \ref{thm:three-gens-one-indispensable}) is an original result.  It provides an upper bound for chromatic numbers of circulant graphs with three generators, one of which is indispensable.  We impose some additional conditions as well, but as we discuss in Section \ref{sec:three-gen}, we lose no generality by doing so.

\begin{Thm}\label{thm:three-gens-one-indispensable}
Let $a,b,c,n$ be positive integers, where $c\mid n$ and $c\nmid a$.  Suppose that $\{a,b,c\}$ generates $\mathbb{Z}_n$, where $c$ is indispensable but $a$ is dispensable.  Let $\alpha, \beta$ be integers such that $a\equiv\alpha b+\beta c\;(\text{mod }n)$.  Then \[\chi(C_n(a,b,c))\leq\begin{cases}
5 & \text{ if }c=5\text{ and }\alpha\equiv\pm 2\;(\text{mod }5)\\
4 & \text{ if }c=13\text{ and }\alpha\equiv\pm 5\;(\text{mod }13)\\
4 & \text{ if }3\nmid c\text{ and }\alpha\equiv\pm 2\;(\text{mod }c)\text{ or }2\alpha\equiv\pm 1\;(\text{mod }c)\\
3 & \text{ otherwise}
\end{cases}\]
\end{Thm}

We prove this theorem in Section \ref{sec:three-gen}.  Usually Theorem \ref{thm:three-gens-one-indispensable} suffices to compute the chromatic number exactly, and we give an example in Section \ref{sec:three-gen} to illustrate this.  Moreover, we find a family of examples for which Theorem \ref{thm:three-gens-one-indispensable} produces the chromatic number, but so far as we are aware no other result in the literature does the same.

Using the same central idea, we also provide alternate proofs of some previously known results.  The first concerns \emph{minimal} generating sets, i.e., sets $S$ for which every element is indispensable.  In \cite{Garcia}, Garc\'ia-Marco and Knauer prove that a Cayley graph of a nilpotent group with a minimal generating set has chromatic number at most $3$.  (We remark that \cite{Barajas-Serra} contains some similar results for so-called ``quasi-minimal'' generating sets for circulant graphs.)  Using the method of Heuberger matrices, we recover the theorem of Garc\'ia-Marco and Knauer for the special case of abelian groups.  We discuss the details in Section \ref{sec:minimal}.

In \cite{Heuberger}, Heuberger furnishes a formula for the chromatic number of an arbitrary circulant graph with two generators.  Here we give an alternate proof that, we hope, may lend itself more readily to generalization.  Our jumping-off point is the observation that, by the aforementioned theorem of Garc\'ia-Marco and Knauer, we may assume that one of the generators is dispensable.  For such circulant graphs, we introduce a family of graph homomorphisms, which we call ``extension ladder homomorphisms.''  These allow us to reduce the problem to a finite set of small graphs, whereupon one can color by exhaustion.  We discuss the details in Section \ref{sec:two-gen}.

\subsection{Preliminaries}

\subsubsection{Circulant graphs}\label{subsubsection-circulants}

We record here a few basic facts we will need about circulant graphs.  These can be found, for example, in \cite{Heuberger}.

\begin{Lem}An abelian Cayley graph $\text{Cay}(G,S)$ is connected if and only if $S$ generates $G$.  In particular, the circulant graph $C_n(a_1,\dots,a_k)$ is connected if and only if $\gcd(a_1,\dots,a_k,n)=1$.\end{Lem}

We remark that if $d=\gcd(a_1,\dots,a_k,n)>1$, then $C_n(a_1,\dots,a_k)$ can be written as a disjoint union of $d$ graphs, each of which is isomorphic to $C_{n/d}(a_1/d,\dots,a_k/d)$, in which case we have that \[\chi(C_n(a_1,\dots,a_k))=\chi(C_{n/d}(a_1/d,\dots,a_k/d)).\]More generally, an abelian Cayley graph $\text{Cay}(G,S)$ is a disjoint union of the graphs $\text{Cay}(G_0,S)$, where $G_0$ is the subgroup of $G$ generated by $S$.  So in our study of chromatic numbers of abelian Cayley graphs (including circulant graphs), we lose nothing by restricting our attention to connected graphs.

\begin{Lem}\label{lemma-circulant-bipartite}
Suppose $X=C_n(a_1,\dots,a_k)$ is a connected circulant graph with $0<a_i<n$ for all $1\leq i\leq k$.  Then $\chi(X)=2$ if and only if $n$ is even and $a_1,\dots,a_k$ are all odd.\end{Lem}

The graph $X$ in Lemma \ref{lemma-circulant-bipartite} would have loops if $a_i$ were divisible by $n$ for some $i$.  The condition $0<a_i<n$ prohibits this annoying special case.  In all other cases one can, without changing $X$, reduce $a_i$ modulo $n$ so that this condition is met.

\subsubsection{Heuberger matrices}
\label{subsection-Heuberger-matrices}

We record here (rather tersely) a few basic facts we will need about Heuberger matrices.  Proofs of these can be found in \cite{Cervantes-Krebs, Cervantes-Krebs-small-cases}.  Although all needed results are briefly summarized below, in our proofs we refer to specific theorems within those articles whenever we invoke them.

For a $1\times 1$ matrix, we have that $(n)^{\text{SACG}}$ is an $n$-cycle if $|n|\geq 3$; a path of length $2$ if $n=\pm 2$; a single vertex with a loop if $n=\pm 1$; and a doubly infinite path graph if $n=0$.

If $M'$ is obtained from $M$ in any of the following ways, then the corresponding graphs are isomorphic: permuting columns; multiplying a column by $-1$; adding an integer multiple of one column to another; permuting rows; multiplying a row by $-1$; deleting a zero column.

Deleting a zero row does not change the chromatic number of the corresponding graph.

If $M'$ is obtained from $M$ in any of the following ways, then there is a graph homomorphism from $M^{\text{SACG}}$ to $(M')^\text{SACG}$: merging two rows by adding them; appending a column; reducing a column by a common factor.  We use the following notation to indicate that such a homomorphism exists:\[M^{\text{SACG}}\xrightarrow{\ocirc} (M')^\text{SACG}.\]If there is a graph homomorphism from a graph $X$ to a graph $Y$, then $\chi(X)\leq\chi(Y)$.

We have that $\chi(M^{\text{SACG}})=2$ if and only if for each column of $M$, the sum of its entries is even.

If $M$ has a block structure --- that is, if \[M=\begin{pmatrix}M_1 & 0\\
0 & M_2\end{pmatrix}\] for some integer matrices $M_1, M_2$ --- then $\chi(M^{\text{SACG}})=\max(\chi(M_1^{\text{SACG}}),\chi(M_2^{\text{SACG}}))$.

We say that an $n$-tuple $(a_1,\dots,a_n)$ of integers is \emph{primitive} if $\text{gcd}(a_1,\dots,a_n)=1$.  We note that such a tuple can be non-primitive either if all of its entries are zero, or otherwise if $\text{gcd}(a_1,\dots,a_n)>1$. 

\section{Main results}

\subsection{A key lemma}

The jumping-off point for each of our main theorems is the following lemma, which shows that in an abelian Cayley graph, a generator is indispensable if and only if its corresponding row in a Heuberger matrix is non-primitive.

\begin{Lem}\label{lem:indispensable-Heuberger}Let $G$ be an abelian group; let $a_1, \dots, a_k$ be elements of $G$ that generate $G$; and let $M$ be a corresponding Heuberger matrix.  Then $a_i$ is indispensable if and only if the $i$th row of $M$ is non-primitive.
\end{Lem}

\begin{proof}  By permuting rows/generators, we may assume without loss of generality that $i=k$.

First suppose that $a_k$ is indispensable and that $r_k$ is not a zero row, where $r_k=(b_{k1}\;\cdots\;b_{kn})$ is the $k$th row of $M$.  Let $d$ be the gcd of the entries of $r_k$.  If $d=1$, then there exist integers $c_1,\dots,c_n$ such that $c_1b_{k1}+\cdots+c_nb_{kn}=1$.  Let $y_1,\dots,y_n$ be the columns of $M$. Then $c_1 y_1^t+\cdots+c_n y_n^t=(w_1\;\cdots\;w_{k-1}\;1)$ for some integers $w_1,\dots,w_{k-1}$.  (The superscript $t$ denotes transpose.)  By the definition of $M$, we get that $a_k=-w_1 a_1-\cdots-w_{k-1} a_{k-1}$, violating indispensability.  Thus $r_k$ is non-primitive.

Conversely, if $a_k$ is dispensable, then there exist integers $w_1,\dots,w_{k-1}$ such that $a_k=w_1 a_1+\cdots+w_{k-1}a_{k-1}$, so the column vector \((-w_1\;\cdots\;-w_{k-1}\;1)^t\) can be expressed as a linear combination of $y_1,\dots,y_n$ with coefficients in $\mathbb{Z}$.  Thus $r_k$ is primitive.\end{proof}

\subsection{Circulant graphs with three generators, one of which is indispensable}\label{sec:three-gen}

In this section we provide an upper bound (as stated in Theorem \ref{thm:three-gens-one-indispensable}) for the chromatic number of a circulant graph with three generators, one of which is indispensable.  Moreover, we show that in many cases this is sufficient to compute the chromatic number exactly.  Prior to proving Theorem \ref{thm:three-gens-one-indispensable}, we first explain how, for an arbitrary graph of this form, it is either isomorphic to one for which the conditions of this theorem are satisfied, or else we can find its chromatic number another way.

Let $a,b,c$ be positive integers, and suppose that $S=\{a,b,c\}$ generates $\mathbb{Z}_n$ and that $c$ is indispensable.  (To avoid notational clutter, we use the same variable to denote an integer as well as its reduction modulo $n$.)  So we have that $\text{gcd}(a,b,c,n)=1$ but $\text{gcd}(a,b,n)>1$.  Consider the circulant graph $C_n(a,b,c)$.

We begin by showing that it suffices to consider such graphs when $c$ divides $n$.  As a first step towards that end, let $w$ be the largest integer which divides $c$ but is relatively prime to $n$.  Taking $w^{-1}$ to be a multiplicative inverse of $w$ modulo $n$, we have that $C_n(a,b,c)\cong C_n(w^{-1}a,w^{-1}b,w^{-1}c)$, as multiplication by $w^{-1}$ simultaneously defines a group as well as graph isomorphism.  Hence $w^{-1}c$ is indispensable in the generating set for the latter graph.  Moreover, $w^{-1}c$ is congruent modulo $n$ to an integer, every prime factor of which also divides $n$.  Replacing $a,b,c$ with $w^{-1}a,w^{-1}b,w^{-1}c$ we see that we may (and now shall) assume our graph $C_n(a,b,c)$ has the property that if a prime $p$ divides $c$, then $p$ divides $n$.  For a prime $p$ and a positive integer $x$, let $\nu_p(x)$ be the largest integer $m$ such that $p^m\mid x$.  Let $\rho$ be the product of all primes $p$ such that $\nu_p(c)\leq \nu_p(n)$.  One can show that $c+\rho n=\psi u$, where $\psi\mid n$ and $\text{gcd}(u,n)=1$.  Replacing $a,b,c$ with $u^{-1}a,u^{-1}b,\psi$ yields the desired result.  In this way we see that we lose no generality by restricting our attention to generating sets $\{a,b,c\}$ for $\mathbb{Z}_n$ in which $c$ is indispensable and $c\mid n$.

Indeed, we may restrict further still.  By Thm. \ref{thm:abelian-minimal-3-colorable}, if $\{a,b,c\}$ is minimal, then $\chi(C_n(a,b,c))\leq 3$, and Lemma \ref{lemma-circulant-bipartite} then allows us to quickly determine whether the chromatic number is $2$ or $3$.  So without loss of generality, we can assume that $a$ is dispensable.  Thus there exist integers $\alpha, \beta$ such that $a\equiv\alpha b+\beta c\;(\text{mod }n)$.

\begin{proof}[Proof of Theorem \ref{thm:three-gens-one-indispensable}]Let $X=C_n(a,b,c)$.  We have the following Heuberger matrix for $X$:\[M=\begin{pmatrix}
1 & 0 & 0\\
-\alpha & c & 0\\
-\beta & -b & \frac{n}{c}
\end{pmatrix},\] as one can show by  direct computation, or else by applying a formula in \cite{Cervantes_2023-preprint}.  (Here we refer to the preprint rather than the published version of this article, because it contains a corrected version of the relevant formula.)

Let $d=\text{gcd}(a,b,n)>1$.  We see that $d$ divides $\beta$, $b$, and $n/c$.  We obtain a graph homomorphism by appending the column $(0,0,d)^t$ to $M$.  We then use column operations to eliminate the bottom three entries of $M$, whereupon we may delete the zero column.  In other words, we have the following:

\[M^\text{SACG}\xrightarrow{\ocirc}\begin{pmatrix}
1 & 0 & 0 & 0\\
-\alpha & c & 0 & 0\\
-\beta & -b & \frac{n}{c} & d
\end{pmatrix}^\text{SACG}=\begin{pmatrix}
1 & 0 & 0 & 0\\
-\alpha & c & 0 & 0\\
0 & 0 & 0 & d
\end{pmatrix}^\text{SACG}= \begin{pmatrix}
1 & 0 & 0\\
-\alpha & c & 0\\
0 & 0 & d
\end{pmatrix}^\text{SACG}.\]

Using results from Section \ref{subsection-Heuberger-matrices}, we then get that \[\chi(X)\leq \max\left\lbrace\chi\left(\begin{pmatrix}
1 & 0\\
-\alpha & c
\end{pmatrix}^\text{SACG}\right),\chi\left((d)^\text{SACG}\right)\right\rbrace.\]  Because $d>1$, we have that $(d)^\text{SACG}$ is either a $d$-cycle (if $d\geq 3$) or an edge (if $d=2$), so $\chi\left((d)^\text{SACG}\right)\leq 3$.  Also, $\begin{pmatrix}
1 & 0\\
-\alpha & c
\end{pmatrix}^\text{SACG}$ is isomorphic to $C_c(1,\alpha)$.  Because $c\mid n$ and $c\nmid a$, we must have that $\alpha\nequiv 0\;(\text{mod }c)$.  The result now follows from Heuberger's formula (Theorem \ref{theorem-Heubergers}) for the chromatic number of a circulant graph with two generators.\end{proof}

We remark that equality may fail in Theorem \ref{thm:three-gens-one-indispensable}.  For example, take $a=2, b=1, c=5, n=15, \alpha=2, \beta=0$.  Theorem \ref{thm:three-gens-one-indispensable} tells us that $\chi(C_n(a,b,c))\leq 5$. However, reduction modulo $3$ gives us a proper $3$-coloring, whereupon by Lemma \ref{lemma-circulant-bipartite} we have that $\chi(C_n(a,b,c))=3$.

More often, though, the upper bound in Theorem \ref{thm:three-gens-one-indispensable} is achieved.  In our next example, we examine two families of circulant graphs for which this occurs.

\begin{Ex}
Let $p$ and $q$ be distinct primes with $q\geq 17$ and $p\geq 5$.  Let $n=pq$.  Let $X= C_n(p, 3p, q)$.  We take $a=p, b=3p, c=q$ and apply Theorem \ref{thm:three-gens-one-indispensable}.  Because $c\geq 17$, neither of the first two cases applies.  Also, $a\equiv \alpha b+\beta c\;(\text{mod}\;n)$ gives us that $p\equiv 3\alpha p\;(\text{mod}\;c)$, which implies that neither $\alpha\equiv\pm 2\;(\text{mod }c)\text{ nor }2\alpha\equiv\pm 1\;(\text{mod }c)$ can hold (again using that $c\geq 17$), so the third case does not apply either.  Hence $\chi(X)\leq 3$.  By Lemma \ref{lemma-circulant-bipartite}, we get that $\chi(X)= 3$.

We now show that other formulae in the literature do not duplicate this computation.  We do not have $\text{gcd}(n,a)=3$ for any generator $a$ in $S=\{p, 3p, q\}$, so the main theorem in \cite{Nicoloso-solo} does not apply.   Thm. 1.3 in \cite{Nicoloso-solo} similarly doesn’t apply, because we cannot have $a+b=c$ or $a+b=n-c$ no matter how we label $a,b,c$.  We note that for $w$ coprime to $n$, the generating set $\{wa,wb,wc\}$ produces a Cayley graph isomorphic to $X$, but the condition $a+b=c$ or $a+b=n-c$ is invariant with respect to such isomorphisms.  Moreover, Thm. 1.4 in \cite{Nicoloso-solo} does not apply, because the condition $n\geq 4bc$ is not met.  Multiplication by $w$ as above cannot produce a new generating set $wS$ which does meet that condition, as one can see by considering that $wS$ must still contain two elements divisible by $p$ and one divisible by $q$.\end{Ex}

\begin{Ex}
Let $p$ and $q$ be distinct primes with $q\geq 7$.  Let $n=pq$.  Let $X= C_n(2p, p, q)$.  We take $a=2p, b=p, c=q, \alpha=2, \beta=0$.  Applying Theorem \ref{thm:three-gens-one-indispensable} gives us $\chi(X)\leq 4$.  But $X$ contains $C_n(2p, p)$ as a subgraph, and $C_n(2p, p)$ is a disjoint union of $p$ copies of $C_q(2, 1)$, which has chromatic number $4$ by Theorem \ref{theorem-Heubergers}.  Therefore $\chi(X)=4$.
\end{Ex}

\subsection{Abelian Cayley graphs with minimal generating sets}\label{sec:minimal}

In this section, we prove the following theorem using the method of Heuberger matrices.

\begin{Thm}[\cite{Garcia}]\label{thm:abelian-minimal-3-colorable}Let $G$ be an abelian group, and suppose that $S$ is a finite minimal generating set for $G$.  Then $\chi(\text{Cay}(G,S))\leq 3.$\end{Thm}

\begin{proof}
Let $M$ be a Heuberger matrix for $X=\text{Cay}(G,S)$.  We may insert zero rows and zero columns without changing the chromatic number, so we may assume that $M$ is an $m\times m$ square matrix.

Let $y_1,\dots,y_m$ be the columns of $M$.  Let $d_i$ be the gcd of the $i$th row of $M$, for all nonzero rows of $M$.  If the $i$th row of $M$ is a zero row, then let $d_i=0$.  By Lemma \ref{lem:indispensable-Heuberger}, we have that $d_i\neq 1$ for all $i$.  For all $j=1,\dots,m$, let $e_j\in\mathbb{Z}^m$ be the column vector with a $1$ in the $j$th position and $0$ elsewhere.  We then have the following:\[M^\text{SACG}\xrightarrow{\ocirc}(y_1\;\;\cdots\;\;y_m\;\;d_1 e_1\;\;\cdots\;\;d_m e_m)^\text{SACG}= (d_1 e_1\;\;\cdots\;\;d_m e_m)^\text{SACG}.\]Here we first obtain a homomorphism by appending $m$ columns.  Then we delete the first $m$ columns of the original matrix; this does not change the associated graph, because those columns are in the $\mathbb{Z}$-span of the last $m$ columns. Therefore \[\chi(X)\leq\chi((d_1 e_1\;\;\cdots\;\;d_m e_m)^\text{SACG}).\]The matrix $(d_1 e_1\;\;\cdots\;\;d_m e_m)$ has a block structure---indeed, it is diagonal---and so by the results of Section \ref{subsection-Heuberger-matrices}, we have that $\chi((d_1 e_1\;\;\cdots\;\;d_m e_m)^\text{SACG})$ equals the maximum of $\chi((d_i)^\text{SACG})$ for $i=1,\dots,m$.  Again using the results of Section \ref{subsection-Heuberger-matrices} as well as the fact that $d_i\neq 1$ for all $i$, we have that $\chi((d_i)^\text{SACG})\leq 3$ for all $i$.  The result follows.\end{proof}

As noted in the introduction, Theorem \ref{thm:abelian-minimal-3-colorable} is a special case of the main theorem in \cite{Garcia}, where it is proved for nilpotent groups.  

For circulant graphs, Theorem \ref{thm:abelian-minimal-3-colorable} can be stated in terms of gcds:

\begin{Cor}\label{cor:circulant-min}Let $a_1,\dots,a_k,n$ be positive integers.  If $\gcd(a_1,\dots,a_k,n)=1$ and $\gcd(a_1,\dots,\hat{a_i},\dots,a_k,n)>1$ for all $i=1,\dots,k$, then $\chi(C_n(a_1,\dots,a_k))\leq 3$.\end{Cor}

The notation $\hat{a_i}$ indicates that all $a_j$ are included except $a_i$.

\begin{Ex}
Cor. \ref{cor:circulant-min} shows that $\chi(C_{30}(6,10,15))\leq 3$.    By Lemma \ref{lemma-circulant-bipartite}, we have that $\chi(C_{30}(6,10,15))\geq 3$.  Therefore $\chi(C_{30}(6,10,15))= 3$.
\end{Ex}

\subsection{Circulant graphs with two generators}\label{sec:two-gen}

The following theorem completely determines the chromatic number of a circulant graph with two generators.  The goal of this section is to provide a new and substantially more intuitive proof of this theorem.

\begin{Thm}[{\cite[Theorem 3]{Heuberger}}]\label{theorem-Heubergers}Let $C_n(a,b)$ be a circulant graph with $0<a<n$ and $0<b<n$, such that $\text{gcd}(a,b,n)=1$.  Then \[\chi(C_n(a,b)) = \begin{cases}

2 & \text{if }a\text{ and }b\text{ are both odd, but }n\text{ is even}
\\

5 & \text{if }n=5\text{ and }a\equiv \pm 2b\Mod 5\\

4 & \text{if }n=13,\text{ and }\; a\equiv \pm 5b\Mod {13}\\

4 & \text{if }(i)\; n\neq 5, \text{ and } (ii)\;3\nmid n,\text{ and }(iii)\; a\equiv \pm 2b\Mod {n}\text{ or  }b\equiv \pm 2a\Mod {n}\\

3 & \text{otherwise.}\end{cases}\]\end{Thm}

\begin{proof}We begin by applying Theorem \ref{thm:abelian-minimal-3-colorable}. This will allow us to restrict our attention to the case where one of the generators is $1$.  To see that this is so, suppose that $C_n(a,b)$ is a connected circulant graph, so that $\text{gcd}(a,b,n)=1$.  By Cor. \ref{cor:circulant-min}, if $\text{gcd}(a,n)>1$ and $\text{gcd}(b,n)>1$, then $\chi(C_n(a,b))\leq 3$.  Without loss of generality, then, we may assume that $\text{gcd}(a,n)=1$.  But then $a$ has a multiplicative inverse $x$ modulo $n$, so multiplication by $x$ defines a graph isomorphism between $C_n(a,b)$ and $C_n(1,bx)$.  We note that the conditions in Theorem \ref{theorem-Heubergers} are invariant under multiplication by $x$.  Thus it suffices to prove the theorem under the assumption that $a=1$.  Moreover, replacing $b$ by $n-b$ if necessary, we may also assume that $0<2b\leq n$.

With those restrictions in place, our overall strategy is as follows.  We wish to find the chromatic numbers of the graphs $C_n(1,b)$.  Lower bounds can be established by standard methods; the main task is to find the correct upper bounds.  Towards this end, we construct a finite auxiliary family of graphs, defined below.  Every graph $C_n(1,b)$, we show, admits a graph homomorphism to a member of our finite auxiliary family of graphs.  Coloring the members of this finite set of graphs, then, provides upper bounds for chromatic numbers of $C_n(1,b)$ for all $b$.  There are a small handful of cases for which this upper bound does not match the previously obtained lower bound, and we can handle those cases individually.

\vspace{.1in}

We now elucidate the details.

First we obtain lower bounds.  By Lemma \ref{lemma-circulant-bipartite}, we have that $\chi(C_n(1,b))=2$ if and only if $n$ is even and $b$ is odd.  From now on we assume $\chi(C_n(1,b))\geq 3$.  Now suppose that $3\nmid n$ and $b\equiv \pm 2\;(\text{mod } n)$.  The vertices $0, 1, 2, 3$ form a ``diamond'' with endpoints $0$ and $3$ --- that is, they induce a  subgraph on those four vertices which would be complete if an edge between $0$ and $3$ were added.  As in \cite{Tim}, we define the ``Diamond is Forever subgroup'' $D(\mathbb{Z}_n,\{\pm 1,\pm 2\})$ to be the subgroup of $\mathbb{Z}_n$ generated by all endpoints of diamonds, one endpoint of which is $0$.  It is straightforward to verify that in any proper $3$-coloring of $C_n(1,2)$, all elements of $D(\mathbb{Z}_n,\{\pm 1,\pm 2\})$ must receive the same color.  However, because $3\nmid n$, we have that $D(\mathbb{Z}_n,\{\pm 1,\pm 2\})=\mathbb{Z}_n$.  In particular $0$ and $1$ receive the same color, a contradiction.  Thus $\chi(C_n(1,b))\geq 4$ in this case.  If $1\equiv \pm 2b\;(\text{mod } n)$, then multiplication by $2$ gives us $C_n(1,b)\cong C_n(1,2)$.  By what we just proved, the latter graph has chromatic number at least $4$.  If $n=\pm 5$ and $1\equiv\pm 2b\;(\text{mod }5)$, then $C_n(1,b)$ is a complete graph on $5$ vertices and so has chromatic number $5$.  In the case when $n=\pm 13$ and $1\equiv\pm 5b\;(\text{mod }13)$, then our graph is isomorphic to $C_{13}(1,5)$, and it can be shown directly that the chromatic number of this graph is $4$---see \cite{Cervantes-Krebs, Heuberger}. 

\vspace{.1in}

Next we establish matching upper bounds.  Towards this end we redraw the graphs $C_n(1,b)$ so that they resemble grids.  This will facilitate the construction of certain graph homomorphisms.  More precisely, we proceed to define a new graph isomorphic to $C_n(1,b)$.

Let $a,r,c\in\Z$ such that $a,c>0$ and $r\geq 0$.  Let $b=r+c$ and $n=ab+r$.  We define a graph $\text{Gr}(a,r,c)$ as follows.  Let \[V(a,r,c) = \{(x,y)\in\mathbb{Z}^2 \mid 0 \leq x < b,\ 0 \leq y < a\} \cup \{(x,a)\in\mathbb{Z}^2 \mid 0 \leq x < r\}.\] The set $V(a,r,c)$ will be the vertex set of $\text{Gr}(a,r,c)$.  Next we define four different sets of edges; the edge set $E(a,r,c)$ of $\text{Gr}(a,r,c)$ will be the union of these four sets.  In these definitions the variables $x$ and $y$ always take values in the integers.

\begin{enumerate}
    \item The set of \emph{grid-like $+1$ edges} is $$\set{\set{(x,y),(x+1,y)} \mid 0\leq x < b-1,\ 0\leq y<a} \cup \set{\set{(x,a),(x+1,a)} \mid 0 \leq x < r-1}.$$
    \item The set of \emph{non-grid-like $+1$ edges} is \[\set{\set{(b-1,y),(0,y+1)} \mid 0 \leq y < a} \cup \set{\set{(r-1,a),(0,0)}}\text{ when }r\neq0, \text{ and}\]\[\set{\set{(b-1,y),(0,y+1)} \mid 0 \leq y < a-1} \cup{\{\{(b-1,a-1),(0,0)\}\}}\text{ when }r=0.\]
    \item The set of \emph{grid-like $+b$ edges} is $$\set{\set{(x,y),(x,y+1)} \mid 0 \leq x < b,\ 0 \leq y < a-1} \cup \set{\set{(x,a-1),(x,a)} \mid 0 \leq x < r}.$$
    \item The set of  \emph{non-grid-like $+b$ edges} is $$\set{\set{(x,a),(x+c,0)} \mid 0 \leq x < r} \cup \set{\set{(x,a-1),(x-r,0)} \mid r \leq x < b}.$$
\end{enumerate}

As an example, the $+1$ and $+8$ edges for $\text{Gr}(4,4,4)$ are shown in Figures \ref{fig:grid-edges} and \ref{fig:grid-edges-2}.  Grid-like edges are either horizontal or vertical, whereas non-grid-like edges are neither.

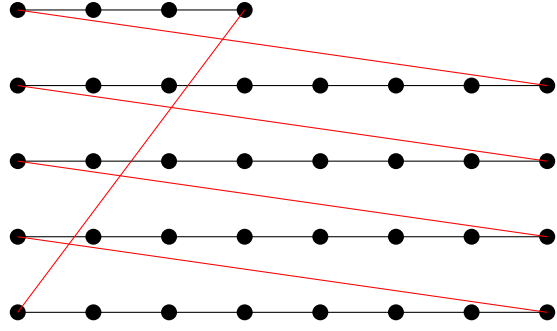
\begin{figure}[ht]
    \centering
    \begin{tikzpicture}
        \foreach \x in {0,1,2,3,4,5,6,7}
        \foreach \y in {0,1,2,3} {\fill (\x,\y) circle (3pt);}
        \foreach \x in {0,1,2,3} {\fill (\x,4) circle (3pt);}
        \foreach \x in {0,1,2,3,4,5,6}
        \foreach \y in {0,1,2,3}{\draw (\x,\y) -- (\x+1,\y);}
        \foreach \x in {0,1,2} {\draw (\x,4) -- (\x+1,4);}
        \foreach \y in {0,1,2,3} {\draw[red] (7,\y) -- (0,\y+1);}
        \draw[red] (3,4) -- (0,0);
    \end{tikzpicture}
    \caption{The $+1$ edges of $\text{Gr}(4,4,4)$}
    \label{fig:grid-edges}\end{figure}

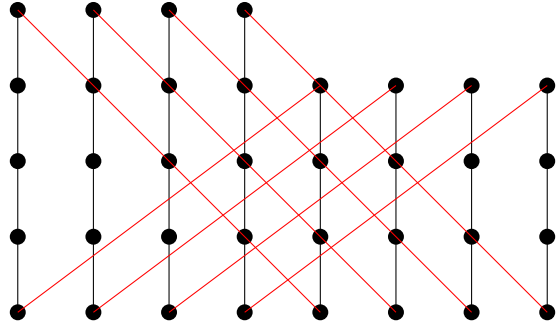
\begin{figure}
    \centering
    \begin{tikzpicture}
        \foreach \x in {0,1,2,3,4,5,6,7}
        \foreach \y in {0,1,2,3} {\fill (\x,\y) circle (3pt);}
        \foreach \x in {0,1,2,3} {\fill (\x,4) circle (3pt);}
        \foreach \x in {0,1,2,3,4,5,6,7}
        \foreach \y in {0,1,2} {\draw (\x,\y) -- (\x,\y+1);}
        \foreach \x in {0,1,2,3} {\draw (\x,3) -- (\x,4);}
        \foreach \x in {0,1,2,3} {\draw[red] (\x,4) -- (\x+4,0);}
        \foreach \x in {4,5,6,7} {\draw[red] (\x,3) -- (\x-4,0);}
    \end{tikzpicture}
\caption{The $+8$ edges of $\text{Gr}(4,4,4)$}
\label{fig:grid-edges-2}
\end{figure}

Consider the circulant graph $C_n(1,b)$, where $b \nmid n$ and $2b < n$. Let $a,r,c \in \N$ such that 
    $n = ab+r$ and $0 \leq r < b$ and $c = b-r$.  (Observe, by the way, that this implies that $a\geq 2$---we'll make use of that later.)  We claim that $C_n(1,b)$ is isomorphic to $\text{Gr}(a,r,c)$.  To see this, for $0\leq k<n$, we map $k\mapsto (x,y)$, where $k=x+yb$ with $0\leq x<b$.  It is straightforward to verify that this mapping sends pairs of adjacent vertices to pairs of adjacent vertices, and pairs of non-adjacent vertices to pairs of non-adjacent vertices.  In particular, $+1$ edges correspond to edges $\{k, k+1\}$ in $C_n(1,b)$, and $+b$ edges correspond to edges $\{k, k+b\}$ in $C_n(1,b)$.

\vspace{.1in}

\subsubsection{The case where $a'\geq 3$}

Next we define some functions between the vertex sets of these graphs.  \[
\begin{array}{l}
    f_{A2}\colon V(a,r,c)\to V(a-2,r,c),\quad f_{A2}(x,y) = \begin{cases}
        (x,y) & 0 \leq y < 2 \\
        (x,y-2) & \text{otherwise}
    \end{cases} \\
    f_{R2}\colon V(a,r,c)\to V(a,r-2,c),\quad f_{R2}(x,y) = \begin{cases}
        (x,y) & 0 \leq x < r-2 \\
        (x-2,y) & \text{otherwise}
    \end{cases} \\
    f_{C2}\colon V(a,r,c)\to V(a,r,c-2),\quad f_{C2}(x,y) = \begin{cases}
        (x,y) & 0 \leq x < r+c-2 \\
        (x-2,y) & \text{otherwise}
    \end{cases}
\end{array}
\]

These three functions are defined when $a\geq 4$, $r\geq 4$, and $c\geq 4$, respectively.  Technically these are not three functions but three families of functions, as there is a dependence on the values of $a$, $r$, and $c$ to define the domains and codomains.  As our notation is cluttered enough already as it is, we saw no need to make it more so by adding three indices.

We call $f_{A2}, f_{R2},$ and $f_{C2}$ \emph{extension ladder homomorphisms}.  The reason for this terminology is that these functions map grid-like edges to grid-like edges; this reminded us of sliding an extension ladder so that some of its rungs are on top of one another.  The $2$ in the subscripts refers to the fact that the functions ``slide'' some vertices by $2$ units.  Later we will define similar functions, sliding by different amounts or in different ways, and we will also call these extension ladder homomorphisms.  Sliding by $2$ units ensures that all grid-like edges map to grid-like edges.  For example, $f_{A2}$ maps the grid-like $+b$ edge $\{(0,1),(0,2)\}$ to the grid-like $+b$ edge $\{(0,1),(0,0)\}$ --- the vertex $(0,1)$ remains in place while the vertex $(0,2)$ slides down two units to $(0,0)$, where it is still adjacent to $(0,1)$.  In most other cases the endpoints of a grid-like edge either both remain fixed or both slide by $2$ units; in either case, they map to another grid-like edge.

To make the extension ladder homomorphisms into honest-to-goodness graph homomorphisms, we must add edges to the graphs $\text{Gr}{(a,r,c)}$.  We define $\text{Gr}^{A2, R2, C2}{(a,r,c)}$ to be the graph with vertex set $V(a,r,c)$ whose edges are the images of edges in $\text{Gr}{(a',r',c')}$ under any possible finite composition of extension ladder homomorphisms $f_{A2}, f_{R2}, f_{C2}$, for some integers $a',r',c'$.

For example, take $a=4, r=4, c=4, a'=8, r'=6, c'=6$.  The non-grid-like $+1$ edge $\{(11,1),(0,2)\}$ in $\text{Gr}{(a',r',c')}$ maps to $\{(7,1),(0,0)\}$ under $f_{A2}^2\circ f_{R2}\circ f_{C2}$.  (We write $g^k$ to mean the function $g$ composed with itself $k$ times.)  So $\{(7,1),(0,0)\}$ is an edge in $\text{Gr}^{A2, R2, C2}{(a,r,c)}$.  Note that $\{(7,1),(0,0)\}$ is not an edge in $\text{Gr}{(a,r,c)}$, however.

We consider the identity function to be a vacuous composition of $f_{A2},f_{R2},f_{C2}$.  Consequently $\text{Gr}{(a,r,c)}$ is a subgraph of $\text{Gr}^{A2, R2, C2}{(a,r,c)}$.

Suppose that $a$ and $a'$ have the same parity; that $r$ and $r'$ have the same parity; that $c$ and $c'$ have the same parity; and that $a'\geq a$ and $r'\geq r$ and $c'\geq c$.  Then $f^{(a'-a)/2}_{A2}\circ f^{(r'-r)/2}_{R2}\circ f^{(c'-c)/2}_{C2}$ is a graph homomorphism from $\text{Gr}(a',r',c')$ to $\text{Gr}^{A2, R2, C2}{(a,r,c)}$.  (Indeed, we ``rigged'' the definition of $\text{Gr}^{A2, R2, C2}{(a,r,c)}$ precisely so that such homomorphisms would exist.)  So for fixed values of $a,r,c$, we have that $\chi(\text{Gr}^{A2, R2, C2}{(a,r,c)})$ provides an upper bound for $\chi(\text{Gr}{(a',r',c')})$ for all such $a',r',c'$.

Generally speaking, most edges in $\text{Gr}^{A2, R2, C2}(a,r,c)$ tend to be in $E(a,r,c)$.  That's partly because grid-like edges map to grid-like edges under extension ladder homomorphisms.  Moreover, one can check that non-grid-like $+1$ edges map to non-grid-like $+1$ edges with one exception: $\{(b-1,1),(0,2)\}$ maps to $\{(b-1,1),(0,0)\}$ under $f_{A2}$.  Thus $\{(b-1,1),(0,0)\}$ is also an edge in $\text{Gr}^{A2, R2, C2}(a,r,c)$.  That leaves only the non-grid-like $+b$ edges; once one accounts for where they can map under compositions of extension ladder homomorphisms, one then knows the entire edge set of $\text{Gr}^{A2, R2, C2}(a,r,c)$.

For example, let us determine the edge set of $\text{Gr}^{A2, R2, C2}(4,4,4)$.  As discussed in the previous paragraph, the key is to find all images of non-grid-like $+b$ edges under compositions of extension ladder homomorphisms.  We consider, for instance, when such an image can have $(0,4)$ as an endpoint.  In $\text{Gr}(4,4,4)$, we have that $\{(0,4),(4,0)\}$ is a non-grid-like $+8$ edge, so the identity mapping gives us $(4,0)$ as such an endpoint.  But $(6,0)$ is another such endpoint; it occurs after applying $f_{C2}$ once.  So are $(0,0)$ and $(2,0)$.  No other such endpoints occur, as can be seen by considering all other finite compositions of $f_{A2}$, $f_{R2}$, and $f_{C2}$.  We then apply the same reasoning to the seven other vertices along the ``top'' of the vertex set, that is, all $(x,y)\in V(4,4,4)$ such that $(x,y+1)\notin V(4,4,4)$.  In this way we find that:

$(0,4)$, $(2,4)$, and $(6,1)$ are adjacent to $(0,0), (2,0), (4,0), (6,0)$.

$(1,4)$, $(3,4)$, and $(7,1)$ are adjacent to $(1,0), (3,0), (5,0), (7,0)$.

$(4,3)$ is adjacent to $(0,0)$.

$(4,4)$ is adjacent to $(1,0)$.

From our earlier discussion of non-grid-like $+1$ edges, we have that $(7,1)$ is adjacent to $(0,0)$.  The remaining edges of $\text{Gr}^{A2, R2, C2}(4,4,4)$ are precisely the elements of $E(4,4,4)$.

\vspace{.1in}

We sometimes wish to consider compositions not of all three extension ladder homomorphisms but only some of them.  We indicate this by writing the subscripts of the extension ladder homomorphisms being composed.  For example, $\text{Gr}^{A2, R2}(a,r,c)$ is defined in the same manner as $\text{Gr}^{A2, R2,C2}(a,r,c)$, except that we will allow only compositions of $f_{A2}$ and $f_{R2}$.

For all integers $a'\geq 3$, $r'\geq 0$, $c'\geq 1$, then, by repeated application of extension ladder homomorphisms, we find that there is a graph homomorphism from $\text{Gr}(a',r',c')$ to a graph in the first column in Table \ref{tab:chi-values-a-at-least-3}.  To find such a graph, look for one in the table whose parameters $a,r,c$ match the parity of $a',r',c'$, respectively.  The superscript of the graph in the table indicates that extension ladder homomorphisms to be iteratively composed to provide the desired graph homomorphism.  The final column of Table \ref{tab:chi-values-a-at-least-3} shows the chromatic number of the corresponding graph.  In this way we obtain upper bounds for the chromatic number of $\text{Gr}(a',r',c')$.

These chromatic numbers were computed by hand, without the assistance of a computer program.  The colorings (as with all colorings of base graphs throughout this section) can be found at \cite{Ahmed-notes}.

Observe from that table that when $n$ is even and $b$ is odd, we get an upper bound of $2$.  Extension ladder homomorphisms preserve the parity of $n$ and $b$, so from this we get $\chi(C_n(1,b))=2$ in such cases.

From Table \ref{tab:chi-values-a-at-least-3} we have $\chi(\text{Gr}^{A2}(3,0,1))=\chi(\text{Gr}^{A2}(4,0,1))=\chi(\text{Gr}^{A2}(3,1,1))=\chi(\text{Gr}^{A2}(4,1,1))=4$.  Thus $\chi(\text{Gr}(a',0,1))\leq 4$ and $\chi(\text{Gr}(a',1,1))\leq 4$ for all $a'\geq 3$.  But $\text{Gr}(a',0,1)\cong C_{2a'}(1,2)$ and $\text{Gr}(a',1,1)\cong C_{2a'+1}(1,2)$.  Recall from our discussion of lower bounds that $\chi(C_n(1,2))\geq 4$ whenever $3\nmid n$.  This gives us that $\chi(C_n(1,2))=4$ whenever $3\nmid n$ and $n\geq 6$.  When $3\mid n$, we observe directly that reduction modulo $3$ provides a $3$-coloring, so $\chi(C_n(1,2))=3$ in that case.

Every other row in Table \ref{tab:chi-values-a-at-least-3} shows a chromatic number of $3$.  In these cases we also have a lower bound of $3$, and thus the chromatic number of all such graphs is $3$.

\begin{table}
\caption{Chromatic numbers of ``base graphs'' for $a\geq 3$}
\label{tab:chi-values-a-at-least-3}
\begin{center}
\begin{tabular}{c|c|c|c|c|c|c}
$a$ & $r$ & $c$ & $b=r+c$ & $n=ab+r$ & Graph $X$ & $\chi(X)$\\
\hline
\hline
3 & 0 & 1 & 1 & 3 & $\text{Gr}^{A2}(a,r,c)$ & 3\\
\hline
3 & 0 & 2 & 2 & 6 & $\text{Gr}^{A2,C2}(a,r,c)$ & 4\\
\hline
3 & 0 & 3 & 3 & 9 & $\text{Gr}^{A2,C2}(a,r,c)$ & 3 \\
\hline
3 & 1 & 1 & 2 & 7 & $\text{Gr}^{A2}(a,r,c)$ & 4 \\
\hline
3 & 1 & 2 & 3 & 10 & $\text{Gr}^{A2,C2}(a,r,c)$ & 2 \\
\hline
3 & 1 & 3 & 4 & 13 & $\text{Gr}^{A2,C2}(a,r,c)$ & 3 \\
\hline
3 & 2 & 1 & 3 & 11 & $\text{Gr}^{A2,R2}(a,r,c)$ & 3 \\
\hline
3 & 2 & 2 & 4 & 14 & $\text{Gr}^{A2,R2,C2}(a,r,c)$ & 3 \\
\hline
3 & 2 & 3 & 5 & 17 & $\text{Gr}^{A2,R2,C2}(a,r,c)$ & 3 \\
\hline
3 & 3 & 1 & 4 & 15 & $\text{Gr}^{A2,R2}(a,r,c)$ & 3 \\
\hline
3 & 3 & 2 & 5 & 18 & $\text{Gr}^{A2,R2,C2}(a,r,c)$ & 2 \\
\hline
3 & 3 & 3 & 6 & 21 & $\text{Gr}^{A2,R2,C2}(a,r,c)$ & 3 \\
\hline
4 & 0 & 1 & 1 & 4 & $\text{Gr}^{A2}(a,r,c)$ & 2 \\
\hline
4 & 0 & 2 & 2 & 8 & $\text{Gr}^{A2,C2}(a,r,c)$ & 4 \\
\hline
4 & 0 & 3 & 3 & 12 & $\text{Gr}^{A2,C2}(a,r,c)$ & 2 \\
\hline
4 & 1 & 1 & 2 & 9 & $\text{Gr}^{A2}(a,r,c)$ & 4 \\
\hline
4 & 1 & 2 & 3 & 13 & $\text{Gr}^{A2,C2}(a,r,c)$ & 3 \\
\hline
4 & 1 & 3 & 4 & 17 & $\text{Gr}^{A2,C2}(a,r,c)$ & 3 \\
\hline
4 & 2 & 1 & 3 & 14 & $\text{Gr}^{A2,R2}(a,r,c)$ & 2 \\
\hline
4 & 2 & 2 & 4 & 18 & $\text{Gr}^{A2,R2,C2}(a,r,c)$ & 3 \\
\hline
4 & 2 & 3 & 5 & 22 & $\text{Gr}^{A2,R2,C2}(a,r,c)$ & 2 \\
\hline
4 & 3 & 1 & 4 & 19 & $\text{Gr}^{A2,R2}(a,r,c)$ & 3 \\
\hline
4 & 3 & 2 & 5 & 23 & $\text{Gr}^{A2,R2,C2}(a,r,c)$ & 3 \\
\hline
4 & 3 & 3 & 6 & 27 & $\text{Gr}^{A2,R2,C2}(a,r,c)$ & 3 \\
\end{tabular}
\end{center}
\end{table}

For example, again consider $\text{Gr}(4,4,4)$, which is isomorphic to $C_{36}(1,8)$.  Applying $f_{R2}\circ f_{C2}$ gives us a graph homomorphism from $\text{Gr}(4,4,4)$ to $\text{Gr}^{A2,R2,C2}(4,2,2)$, which appears in Table \ref{tab:chi-values-a-at-least-3}.  The graph $\text{Gr}^{A2,R2,C2}(4,2,2)$ is the same as $\text{Gr}(4,2,2)$ except that there are also edges between $(0,4)$ and $(0,0)$; $(1,4)$ and $(1,0)$; $(2,3)$ and $(2,0)$; $(3,3)$ and $(3,0)$; and $(0,0)$ and $(3,1)$.  One can properly $3$-color $\text{Gr}^{A2,R2,C2}(4,2,2)$ as follows:$$\begin{matrix}
    (0,4)\mapsto \text{Y} & (1,4)\mapsto\text{R} & \; & \;\\
    (0,3)\mapsto \text{R} & (1,3)\mapsto \text{B} & (2,3)\mapsto \text{Y} & (3,3)\mapsto \text{B}\\
     (0,2)\mapsto \text{B} & (1,2)\mapsto \text{Y} & (2,2)\mapsto \text{R} & (3,2)\mapsto \text{Y}\\
      (0,1)\mapsto \text{R} & (1,1)\mapsto \text{B} & (2,1)\mapsto \text{Y} & (3,1)\mapsto \text{R}\\

      (0,0)\mapsto \text{B} & (1,0)\mapsto \text{Y} & (2,0)\mapsto \text{B} & (3,0)\mapsto \text{Y}
\end{matrix}$$where R, Y, and B stand for red, yellow, and blue, respectively.  Pulling this coloring back via $f_{R2}\circ f_{C2}$ gives us the following proper coloring of $\text{Gr}(4,4,4)$:$$\begin{matrix}
    (0,4)\mapsto \text{Y} & (1,4)\mapsto\text{R} & (2,4)\mapsto \text{Y} & (3,4)\mapsto\text{R} & \; & \;  & \; & \;\\
    (0,3)\mapsto \text{R} & (1,3)\mapsto \text{B} & (2,3)\mapsto \text{R} & (3,3)\mapsto \text{B} & (4,3)\mapsto \text{Y} & (5,3)\mapsto \text{B} & (6,3)\mapsto \text{Y} & (7,3)\mapsto \text{B}\\
     (0,2)\mapsto \text{B} & (1,2)\mapsto \text{Y} & (2,2)\mapsto \text{B} & (3,2)\mapsto \text{Y} & (4,2)\mapsto \text{R} & (5,2)\mapsto \text{Y} & (6,2)\mapsto \text{R} & (7,2)\mapsto \text{Y}\\
      (0,1)\mapsto \text{R} & (1,1)\mapsto \text{B} & (2,1)\mapsto \text{R} & (3,1)\mapsto \text{B} & (4,1)\mapsto \text{Y} & (5,1)\mapsto \text{R}& (6,1)\mapsto \text{Y} & (7,1)\mapsto \text{R}\\

      (0,0)\mapsto \text{B} & (1,0)\mapsto \text{Y} & (2,0)\mapsto \text{B} & (3,0)\mapsto \text{Y} &(4,0)\mapsto \text{B} & (5,0)\mapsto \text{Y} &(6,0)\mapsto \text{B} & (7,0)\mapsto \text{Y}
\end{matrix}$$

  Effectively, the first two ``columns'' are repeated one time due to $f_{R2}$, and the last two are repeated one time due to $f_{C2}$.  Finally, identifying $\text{Gr}(4,4,4)$ with $C_{36}(1,8)$ gives us the proper $3$-coloring $$\text{BYBYBYBY\;RBRBYRYR\;BYBYRYRY\;RBRBYBYB\;YRYR}.$$  This notation means that $0\mapsto\text{B}, 1\mapsto\text{Y},\dots,35\mapsto\text{R}$.  The spacing reflects the ``rows'' in $\text{Gr}(4,4,4)$.

To recap what we have so far: Let $n$ and $b$ be positive integers with $3b\leq n$.  (This corresponds to the requirement that $a'\geq 3$.)  If $n$ is even and $b$ is odd, then $\chi(C_n(1,b))=2$.  If $b=2$ and $3\nmid n$, then $\chi(C_n(1,b))=4$.  Otherwise, $\chi(C_n(1,b))=3$.

\vspace{.1in}

\subsubsection{The case where $a'=2$}

As mentioned previously, our starting assumption that $b\leq n/2$ is equivalent to $a'\geq 2$.  What's left is to contend with the case where $a'=2$.

We first dispatch with the situation in which $r'$ is even.
 When $r'$ is even and $c'$ is odd, we have that $\text{Gr}(2,r',c')\cong C_{3r'+2c'}(1,r'+c')$, which has chromatic number $2$ by \cref{lemma-circulant-bipartite}. When $r'$ and $c'$ are both even and $r'\geq 2$, then there is a graph homomorphism from $\text{Gr}(2,r',c')$ to $\text{Gr}^{R2,C2}(2,2,2)$.  Similarly, when $c'$ is even and $c'\geq 4$, there is a graph homomorphism from $\text{Gr}(2,0,c')$ to $\text{Gr}^{C2}(2,0,4)$.  One can compute by hand that $\text{Gr}^{R2,C2}(2,2,2)$ and $\text{Gr}^{C2}(2,0,4)$ both have chromatic number $3$.  In the special case $r'=0$ and $c'=2$, we have $b=2$ and $3\nmid n=4$, and indeed $\text{Gr}(2,0,2)\cong K_4$, which has chromatic number $4$.  Going forward, then, we assume $r'$ is odd. (Nevertheless, some of our base graphs have $r$ even, as there may be extension ladder homomorphisms to such graphs from graphs with $r'$ odd; this can occur when we ``slide'' by odd amounts.)

By Lemma \ref{lemma-circulant-bipartite}, if $r'$ is odd, then $\chi(\text{Gr}(2,r',c'))\geq 3$.

When $r'$ is odd, we require a slight adjustment to our approach so far.  That's because then the graph $\text{Gr}^{R2, C2}(2,r',c')$ contains a $4$-clique at the vertices $(b-2,0), (b-1,0), (b-2,1), (b-1,1)$.  Hence $\chi(\text{Gr}^{R2, C2}(2,r',c'))\geq 4$ in those cases, but this gives an upper bound of at least $4$ for many graphs $\text{Gr}(2,r',c')$ whose chromatic number is in fact $3$.  To remedy this situation, we define new extension ladder homomorphisms which may ``slide'' by amounts other than $2$.  Moreover, we require one additional parameter to control the position at which the sliding occurs.  Because the value of $a$ remains constant at $2$ for these, we need only slide $r$ and $c$.

To be precise, let $\lambda\geq 2$ and $p\geq \lambda$ be integers.  We define the functions $f_{R(\lambda,p)}$ and $f_{C(\lambda,p)}$ as follows.

\[
\begin{array}{l}
    f_{R(\lambda,p)}\colon V(2,r,c)\to V(2,r-\lambda,c),\quad f_{R(\lambda,p)}(x,y) = \begin{cases}
        (x,y) & 0 \leq x < r-p \\
        (x-\lambda,y) & \text{otherwise}
    \end{cases} \\
    f_{C(\lambda,p)}\colon V(2,r,c)\to V(2,r,c-\lambda),\quad f_{C(\lambda,p)}(x,y) = \begin{cases}
        (x,y) & 0 \leq x < r+c-p \\
        (x-\lambda,y) & \text{otherwise}
    \end{cases}
\end{array}
\]

We have that $f_{R(\lambda,p)}$ is defined whenever $r\geq \lambda+p$, and $f_{C(\lambda,p)}$ is defined whenever $c\geq \lambda+p$.  Observe that the functions $f_{R2}$ and $f_{C2}$ from earlier are precisely $f_{R(2,2)}$ and $f_{C(2,2)}$, respectively.

The graphs $\text{Gr}^{R(\lambda,p),C(\lambda,p)}(2,r,c)$ as well as $\text{Gr}^{R(\lambda,p)}(2,r,c)$ and $\text{Gr}^{C(\lambda,p)}(2,r,c)$ are defined analogously to $\text{Gr}^{A2,R2,C2}(a,r,c)$ and its cousins.

Regrettably, when $\lambda>2$, the extension ladder homomorphisms $f_{R(\lambda,p)}$ and $f_{C(\lambda,p)}$ do not always map grid-like edges to grid-like edges.  For example, for $r=7$, we have that $f_{R(3,3)}$ sends the grid-like edge $\{(3,2),(4,2)\}$ to the non-grid-like edge $\{(3,2),(1,2)\}$.  However, this creates only three edges of $\text{Gr}^{R(3,3),C(3,3)}(2,r,c)$ not in $E(2,r,c)$, namely $\{(r-1,0),(r-3,0)\}$, $\{(r-1,1),(r-3,1)\}$, and $\{(r-1,2),(r-3,2)\}$.  Similarly, $f_{C(3,3)}$ creates only two additional edges in this way, namely $\{(b-3,0),(b-1,0)\}$ and $\{(b-3,1),(b-1,1)\}$.  So the damage is limited.

Most extension ladder homomorphisms we actually used ``slide'' by $3$.  Thus we find it useful to divvy into nine cases according to the congruence classes of $r'$ and $c'$ modulo $3$.  We begin with the more direct cases and work our way up to those that are more involved.

\paragraph{The case where $r'\equiv c'\equiv 0\;(\text{mod }3)$.}  For all such graphs, there is a graph homomorphism from $\text{Gr}(2,r',c')$ to $\text{Gr}^{R(3,3),C(3,3)}(2,3,3)$, which has chromatic number $3$, as one can compute by hand.  (Recall that we can assume $r'$ is odd, so we needn't consider the case where $r'=0$.  And $c'=0$ cannot occur, as $c'$ is strictly positive.)

\paragraph{The case where $r'\equiv0 \;(\text{mod }3)\text{ and } c'\equiv 1\;(\text{mod }3)$.}  In these cases there is a homomorphism from $\text{Gr}(2,r',c')$ to one of the base graphs in Table \ref{tab:chi-values-a-is-2-r-0-c-1}.  All of those base graphs have chromatic number $3$.  Hence $\chi(\text{Gr}(2,r',c'))=3$ for all graphs under consideration in this sub-case.

\begin{table}
\caption{Chromatic numbers of ``base graphs'' for $a=2$ when $r'\equiv0 \;(\text{mod }3)\text{ and } c'\equiv 1\;(\text{mod }3)$}

\label{tab:chi-values-a-is-2-r-0-c-1}
\begin{center}
\begin{tabular}{c|c|c|c|c|c|c}
$a$ & $r$ & $c$ & $b=r+c$ & $n=ab+r$ & Graph $X$ & $\chi(X)$\\
\hline
\hline
2 & 3 & 1 & 4 & 11 & $\text{Gr}^{R(3,3)}(a,r,c)$ & 3\\
\hline
2 & 3 & 4 & 7 & 17 & $\text{Gr}(a,r,c)$ & 3\\
\hline
2 & 9 & 4 & 13 & 35 & $\text{Gr}^{R(6,6)}(a,r,c)$ & 3\\
\hline
2 & 3 & 7 & 10 & 23 & $\text{Gr}^{C(6,6)}(a,r,c)$ & 3\\
\hline
2 & 9 & 7 & 16 & 41 & $\text{Gr}^{R(6,6),C(6,6)}(a,r,c)$ & 3\\
\hline
2 & 3 & 10 & 13 & 29 & $\text{Gr}^{C(6,6)}(a,r,c)$ & 3\\
\hline
2 & 9 & 10 & 19 & 47 & $\text{Gr}^{R(6,6),C(6,6)}(a,r,c)$ & 3\\
\end{tabular}
\end{center}
\end{table}

\paragraph{The case where $r'\nequiv 0 \;(\text{mod }3)\text{ and } c'\equiv 0\;(\text{mod }3)$.}  In these cases there is a homomorphism from $\text{Gr}(2,r',c')$ to one of the base graphs in Table \ref{tab:chi-values-a-is-2-r-1-2-c-0}.  All of those base graphs have chromatic number $3$.  Hence $\chi(\text{Gr}(2,r',c'))=3$ for all graphs under consideration in this sub-case.

\begin{table}
\caption{Chromatic numbers of ``base graphs'' for $a=2$ when $r'\nequiv 0 \;(\text{mod }3)\text{ and } c'\equiv 0\;(\text{mod }3)$}
\label{tab:chi-values-a-is-2-r-1-2-c-0}
\begin{center}
\begin{tabular}{c|c|c|c|c|c|c}
$a$ & $r$ & $c$ & $b=r+c$ & $n=ab+r$ & Graph $X$ & $\chi(X)$\\
\hline
\hline
2 & 1 & 3 & 4 & 9 & $\text{Gr}^{C(3,3)}(a,r,c)$ & 3\\
\hline
2 & 4 & 3 & 7 & 16 & $\text{Gr}^{R(3,3),C(3,3)}(a,r,c)$ & 3\\
\hline
2 & 5 & 3 & 8 & 21 & $\text{Gr}^{R(3,3),C(3,3)}(a,r,c)$ & 3\\
\end{tabular}
\end{center}
\end{table}

\paragraph{The case where $r'\equiv 2 \;(\text{mod }3)\text{ and } c'\nequiv 0\;(\text{mod }3)$.}  In these cases there is a homomorphism from $\text{Gr}(2,r',c')$ to one of the base graphs in Table \ref{tab:chi-values-a-is-2-r-2-c-1-2}.  (Recall that we may assume $r'$ is odd, so we needn't consider $r'=2$.)    All of those base graphs have chromatic number $3$.  Hence $\chi(\text{Gr}(2,r',c'))=3$ for all graphs under consideration in this sub-case.

\begin{table}
\caption{Chromatic numbers of ``base graphs'' for $a=2$ when $r'\equiv 2 \;(\text{mod }3)\text{ and } c'\nequiv 0\;(\text{mod }3)$}
\label{tab:chi-values-a-is-2-r-2-c-1-2}
\begin{center}
\begin{tabular}{c|c|c|c|c|c|c}
$a$ & $r$ & $c$ & $b=r+c$ & $n=ab+r$ & Graph $X$ & $\chi(X)$\\
\hline
\hline
2 & 5 & 1 & 6 & 17 & $\text{Gr}^{R(3,3)}(a,r,c)$ & 3\\
\hline
2 & 5 & 2 & 7 & 19 & $\text{Gr}^{R(3,3)}(a,r,c)$ & 3\\
\hline
2 & 5 & 4 & 9 & 23 & $\text{Gr}^{R(3,3),C(3,3)}(a,r,c)$ & 3\\
\hline
2 & 5 & 5 & 10 & 25 & $\text{Gr}^{R(3,3),C(3,3)}(a,r,c)$ & 3\\
\end{tabular}
\end{center}
\end{table}

\paragraph{The case where $r'\equiv0 \;(\text{mod }3)\text{ and } c'\equiv 2\;(\text{mod }3)$.}

One of the exceptional cases in \cref{theorem-Heubergers} is $C_{13}(1,5)$, which is isomorphic to $\text{Gr}(2,3,2)$.  The graph $C_{13}(1,5)$ has chromatic number $4$.  Repeatedly applying $f_{R(3,3)}$ to get a graph homomorphism to $\text{Gr}^{R(3,3)}(2,3,2)$, then, cannot give us any upper bound better than $4$ for the chromatic number, but it turns out that we can do better than that.  This compelled us to ``back up a few steps'' and put other base graphs in the table.

In all other cases, there is a homomorphism from $\text{Gr}(2,r',c')$ to one of the $3$-chromatic base graphs in Table \ref{tab:chi-values-a-is-2-r-0-c-2}.

\begin{table}
\caption{Chromatic numbers of ``base graphs'' for $a=2$ when $r'\equiv0 \;(\text{mod }3)\text{ and } c'\equiv 2\;(\text{mod }3)$}
\label{tab:chi-values-a-is-2-r-0-c-2}
\begin{center}
\begin{tabular}{c|c|c|c|c|c|c}
$a$ & $r$ & $c$ & $b=r+c$ & $n=ab+r$ & Graph $X$ & $\chi(X)$\\
\hline
\hline
2 & 3 & 2 & 5 & 13 & $\text{Gr}(a,r,c)\cong C_{13}(1,5)$ & 4\\
\hline
2 & 3 & 5 & 8 & 19 & $\text{Gr}^{C(5,5)}(a,r,c)$ & 3\\
\hline
2 & 3 & 8 & 11 & 25 & $\text{Gr}^{C(5,5)}(a,r,c)$ & 3\\
\hline
2 & 3 & 11 & 14 & 31 & $\text{Gr}^{C(5,5)}(a,r,c)$ & 3\\
\hline
2 & 3 & 14 & 17 & 37 & $\text{Gr}^{C(5,5)}(a,r,c)$ & 3\\
\hline
2 & 3 & 17 & 20 & 43 & $\text{Gr}^{C(5,5)}(a,r,c)$ & 3\\
\hline
2 & 9 & 2 & 11 & 31 & $\text{Gr}^{R(6,6)}(a,r,c)$ & 3\\
\hline
2 & 9 & 5 & 14 & 37 & $\text{Gr}^{R(6,6)}(a,r,c)$ & 3\\
\hline
2 & 9 & 8 & 17 & 43 & $\text{Gr}^{R(6,6),C(6,6)}(a,r,c)$ & 3\\
\hline
2 & 9 & 11 & 20 & 49 & $\text{Gr}^{R(6,6),C(6,6)}(a,r,c)$ & 3\\
\end{tabular}
\end{center}
\end{table}

\paragraph{The case where $r'\equiv c'\equiv 1\;(\text{mod }3)$.}  First we discuss a few exceptional cases.

Observe that $\text{Gr}(2,1,1)$ is isomorphic to $C_5(1,2)$.  This is a complete graph on $5$ vertices and thus has chromatic number $5$.

Whenever $r'=1$, we have that $\text{Gr}(2,1,c')\cong C_{3+2c'}(1,1+c')\cong C_{3+2c'}(1,2)$.  The latter isomorphism is obtained by multiplying by $2$.  We have previously shown that $\chi(C_{3+2c'}(1,2))\geq 4$ in this case, because $3\nmid 3+2c'$.  For the matching upper bound, observe that there is a graph homomorphism to $\text{Gr}^{C(3,3)}(2,1,4)$ provided $c'\geq 4$.  The correct result follows.

For graphs of the form $\text{Gr}(2,4+3k,1)$ for some nonnegative integer $k$, there is a graph homomorphism to $\text{Gr}^{R(3,3)}(2,4,1)$, which has chromatic number $3$, and thus $\chi(\text{Gr}(2,4+3k,1))=3$.

These three possibilities correspond to the three base graphs shown in Table \ref{tab:chi-values-a-is-2-r-1-c-1}.

It remains to deal with graphs $\text{Gr}(2,4+3k,4+3\ell)$ for some nonnegative integers $k,\ell$.  We shall show that all such graphs have chromatic number $3$.  Let $q$ and $s$ be integers such that \[4+3\ell=q(3k+6)+s\text{ and }0\leq s<3k+6=r+2.\]We divide into sub-cases according to whether (i) $s=1$, or (ii) $s$ is even, or (iii) $s$ is odd and $s>1$.

\guillemetright\,\guillemetright\, Sub-case (i): $s=1$.

We apply the following function from $V(2,4+3k,4+3\ell)$ to $V(2,5,6)$:\[\phi=\left(f^{R(2,3)}\right)^{\frac32 k-\frac12}\circ f^{C\left(2,\frac32 k+\frac52\right)}\circ f^{C\left(2,\frac32 k+\frac12\right)}\cdots\circ f^{C\left(2,3\right)}\circ\left(f^{C(3k+6,3k+7)}\right)^q.\](Within the ellipses, the second argument decreases by $2$ each time.  Writing a function to a power indicates that it should be composed with itself that many times.)

Tautologically, $\phi$ defines a graph homomorphism from $\text{Gr}(2,4+3k,4+3\ell)$ to the graph $Y$ with vertex set $V(2,5,6)$ whose edges are precisely the images of edges under $\phi$.  It is straightforward to produce a proper $3$-coloring of $Y$.  Hence $\chi(\text{Gr}(2,4+3k,4+3\ell))=3$.  All such colorings of base graphs for this paper, such as $Y$, can be found at \cite{Ahmed-notes}.

\guillemetright\,\guillemetright\,Sub-case (ii): $s$ is even.

This time we use the function\[\phi=\left(f^{R(2,5)}\right)^{(3k-s+5)/2}\circ\left(f^{R(2,3)}\right)^{(s-2)/2}\circ f^{C\left(2,\frac32 k+\frac{11}2\right)}\circ f^{C\left(2,\frac32 k+\frac72\right)}\cdots\circ f^{C\left(2,3\right)}\circ\left(f^{C(3k+6,3k+10)}\right)^q\circ \left(f^{C(2,2)}\right)^{(s-2)/2}.\]with codomain $V(2,5,9)$.  The rest is similar to sub-case (i).

\guillemetright\,\guillemetright\,Sub-case (iii): $s$ is odd and $s>1$.

This time we use the function\[\phi=\left(f^{R(2,5)}\right)^{(3k-s+6)/2}\circ\left(f^{R(2,3)}\right)^{(s-3)/2}\circ f^{C\left(2,\frac32 k+\frac{13}2\right)}\circ f^{C\left(2,\frac32 k+\frac92\right)}\cdots\circ f^{C\left(2,3\right)}\circ\left(f^{C(3k+6,3k+11)}\right)^q\circ \left(f^{C(2,2)}\right)^{(s-3)/2}.\]with codomain $V(2,5,10)$.  The rest is similar to sub-case (i).

\begin{table}
\caption{Chromatic numbers of ``base graphs'' for $a=2$ when $r'\equiv c'\equiv 1\;(\text{mod }3)$}
\label{tab:chi-values-a-is-2-r-1-c-1}
\begin{center}
\begin{tabular}{c|c|c|c|c|c|c}
$a$ & $r$ & $c$ & $b=r+c$ & $n=ab+r$ & Graph $X$ & $\chi(X)$\\
\hline
\hline
2 & 1 & 1 & 2 & 5 & $\text{Gr}(a,r,c)\cong K_5$ & 5\\
\hline
2 & 1 & 4 & 5 & 11 & $\text{Gr}^{C(3,3)}(a,r,c)$ & 4\\
\hline
2 & 4 & 1 & 5 & 14 & $\text{Gr}^{R(3,3)}(a,r,c)$ & 3\\
\end{tabular}
\end{center}
\end{table}

\paragraph{The case where $r'\equiv1 \;(\text{mod }3)\text{ and } c'\equiv 2\;(\text{mod }3)$.}

The graphs $\text{Gr}(2,1,2)$ and $\text{Gr}(2,1,5)$ fall into the general family of graphs of the form $\text{Gr}(2,1,c')$, which we discussed in the case where $r'\equiv c'\equiv1 \;(\text{mod }3)$; it works the same way here.  For graphs of the form $\text{Gr}(2,4+3k,2)$ for some nonnegative integer $k$, there is a graph homomorphism to $\text{Gr}^{R(3,3)}(2,4,2)$, which has chromatic number $3$, and thus $\chi(\text{Gr}(2,4+3k,2))=3$.  These possibilities correspond to the three base graphs shown in Table \ref{tab:chi-values-a-is-2-r-1-c-2}.

\begin{table}
\caption{Chromatic numbers of ``base graphs'' for $a=2$ when $r'\equiv 1 \;(\text{mod }3)\text{ and } c'\equiv 2\;(\text{mod }3)$}
\label{tab:chi-values-a-is-2-r-1-c-2}
\begin{center}
\begin{tabular}{c|c|c|c|c|c|c}
$a$ & $r$ & $c$ & $b=r+c$ & $n=ab+r$ & Graph $X$ & $\chi(X)$\\
\hline
\hline
2 & 1 & 2 & 3 & 7 & $\text{Gr}(a,r,c)$ & 4\\
\hline
2 & 1 & 5 & 6 & 13 & $\text{Gr}^{C(3,3)}(a,r,c)$ & 4\\
\hline
2 & 4 & 2 & 6 & 16 & $\text{Gr}^{R(3,3)}(a,r,c)$ & 3\\
\end{tabular}
\end{center}
\end{table}

We can show that all the remaining graphs $\text{Gr}(2,4+3k,5+3\ell)$, where $k$ and $\ell$ are nonnegative integers have chromatic number $3$ by using the same method as in the previous case involving graphs $\text{Gr}(2,4+3k,4+3\ell)$.  We divvy into more or less the same sub-cases and apply more or less the same compositions of extension ladder homomorphisms, \emph{mutatis mutandis}.  Details can be found at \cite{Ahmed-notes}.\end{proof}

\begin{Rmk}
    In \cite{Heuberger}, explicit colorings of the graphs $C_n(a,b)$ are given, but without indication as to how these were obtained.  By contrast, with this approach, where the coloring comes from becomes transparent.
    We first properly color one of the finitely many base graphs and then pull back by extension ladder homomorphisms.  The main point is that in this way, one can \emph{find} these colorings, not just verify that they work.  (An exception occurs with the more complicated composition of functions in the cases where $r'\equiv 1\;(\text{mod }3)$ and $c'\nequiv 0\;(\text{mod }3)$.  Those we found by coloring many graphs and then looking for patterns.  So there are some limitations to the utility of the method.)  We note that our colorings have a similar flavor to those in \cite{Heuberger} insofar as they often consist of a ``body'' of some repeating string of colors followed by a small ``tail.''  We hope that the method presented in this paper provides some insight into how one can construct these colorings and why they then have the character they do.  In future research, we plan to attempt to extend this technique to the case of circulant graphs (or more generally, abelian Cayley graphs) with three generators. 
\end{Rmk}

\section*{Acknowledgments}

The third author was supported by funding from Cal-Bridge.  The eighth author was supported by an Edison STEM-NET Student Research Fellowship.

\bibliographystyle{amsplain}
\bibliography{references}

@preamble{"\providecommand{\MR}[1]{}"}

@misc{Ahmed-notes,
  author = {{Ferdous Ahmed, David Asraf, David Bonds, Jonathan Davidson, Yunhee Jang, Mike
Krebs, Anand Prakash, and Edgar Yak-De Padua}},
  title = {Authors' notes: Chromatic numbers of circulants with indispensable generators},
  howpublished = {\url{https://www.calstatela.edu/research/mkrebs}},
}

@article {Barajas-Serra,
    AUTHOR = {Barajas, Javier and Serra, Oriol},
     TITLE = {On the chromatic number of circulant graphs},
   JOURNAL = {Discrete Math.},
  FJOURNAL = {Discrete Mathematics},
    VOLUME = {309},
      YEAR = {2009},
    NUMBER = {18},
     PAGES = {5687--5696},
      ISSN = {0012-365X},
   MRCLASS = {05C15},
  MRNUMBER = {2567972},
       DOI = {10.1016/j.disc.2008.04.041},
       URL = {https://doi.org/10.1016/j.disc.2008.04.041},
}

@article {Cervantes-Krebs,
    AUTHOR = {Cervantes, Jonathan and Krebs, Mike},
     TITLE = {Chromatic numbers of {C}ayley graphs of abelian groups: {A}
              matrix method},
   JOURNAL = {Linear Algebra Appl.},
  FJOURNAL = {Linear Algebra and its Applications},
    VOLUME = {676},
      YEAR = {2023},
     PAGES = {277--295},
      ISSN = {0024-3795,1873-1856},
   MRCLASS = {05C15},
  MRNUMBER = {4622039},
       DOI = {10.1016/j.laa.2023.07.016},
       URL = {https://doi.org/10.1016/j.laa.2023.07.016},
}

@unpublished{Cervantes_2023-preprint,
   title={Chromatic numbers of {C}ayley graphs of abelian groups: A matrix method},
   note={\url{https://arxiv.org/abs/2303.06262}},
   author={Cervantes, Jonathan and Krebs, Mike},
   year={2023}
   }

@misc{Cervantes-Krebs-small-cases,
    note = {\url{https://arxiv.org/abs/2303.06272}},
    year = 2023,
    author = {Jonathan Cervantes and Mike Krebs},
title={Chromatic numbers of {C}ayley graphs of abelian groups: Cases of small dimension and rank}, 
}

@article {Garcia,
    AUTHOR = {Garc\'ia-Marco, Ignacio and Knauer, Kolja},
     TITLE = {Coloring minimal {C}ayley graphs},
   JOURNAL = {European J. Combin.},
  FJOURNAL = {European Journal of Combinatorics},
    VOLUME = {125},
      YEAR = {2025},
     PAGES = {Paper No. 104108, 9},
      ISSN = {0195-6698,1095-9971},
   MRCLASS = {05C15 (05C25)},
  MRNUMBER = {4842625},
       DOI = {10.1016/j.ejc.2024.104108},
       URL = {https://doi.org/10.1016/j.ejc.2024.104108},
}

@mastersthesis{Tim,
  author = "Timothy Harris",
  title = "STANDARDIZED ABELIAN {C}AYLEY GRAPHS AND THEIR CHROMATIC NUMBERS IN FOUR DIMENSIONS",
  school = "California State University, Los Angeles",
  type = "M.{S}. thesis",
  year = "2024",
  url = "https://www.calstatela.edu/faculty/mike-krebs"
}

@article {Heuberger,
    AUTHOR = {Heuberger, Clemens},
     TITLE = {On planarity and colorability of circulant graphs},
   JOURNAL = {Discrete Mathematics},
    VOLUME = {268},
      YEAR = {2003},
    NUMBER = {1},
     PAGES = {153–169},
      ISSN = {0012-365X},
       DOI = {https://doi.org/10.1016/S0012-365X(02)00685-4},
       URL = {https://www.sciencedirect.com/science/article/pii/S0012365X02006854}
}

@article {Ilic,
    AUTHOR = {Ili\'{c}, Aleksandar and Ba\v{s}i\'{c}, Milan},
     TITLE = {On the chromatic number of integral circulant graphs},
   JOURNAL = {Comput. Math. Appl.},
  FJOURNAL = {Computers \& Mathematics with Applications. An International
              Journal},
    VOLUME = {60},
      YEAR = {2010},
    NUMBER = {1},
     PAGES = {144--150},
      ISSN = {0898-1221},
   MRCLASS = {05C15},
  MRNUMBER = {2651892},
MRREVIEWER = {Oriol Serra},
       DOI = {10.1016/j.camwa.2010.04.041},
       URL = {https://doi.org/10.1016/j.camwa.2010.04.041},
}

@article {Meszka,
    AUTHOR = {Meszka, Mariusz and Nedela, Roman and Rosa, Alexander},
     TITLE = {The chromatic number of 5-valent circulants},
   JOURNAL = {Discrete Math.},
  FJOURNAL = {Discrete Mathematics},
    VOLUME = {308},
      YEAR = {2008},
    NUMBER = {24},
     PAGES = {6269--6284},
      ISSN = {0012-365X},
   MRCLASS = {05C15 (05C25)},
  MRNUMBER = {2464916},
MRREVIEWER = {Clemens Heuberger},
       DOI = {10.1016/j.disc.2007.11.065},
       URL = {https://doi.org/10.1016/j.disc.2007.11.065},
}

@article {Nicoloso,
    AUTHOR = {Nicoloso, S. and Pietropaoli, U.},
     TITLE = {Vertex-colouring of 3-chromatic circulant graphs},
   JOURNAL = {Discrete Appl. Math.},
  FJOURNAL = {Discrete Applied Mathematics. The Journal of Combinatorial
              Algorithms, Informatics and Computational Sciences},
    VOLUME = {229},
      YEAR = {2017},
     PAGES = {121--138},
      ISSN = {0166-218X},
   MRCLASS = {05C15 (05C25)},
  MRNUMBER = {3679065},
MRREVIEWER = {B. Naimeh Onagh},
       DOI = {10.1016/j.dam.2017.05.013},
       URL = {https://doi.org/10.1016/j.dam.2017.05.013},
}

@misc{Nicoloso-solo,
note = {\url{https://ssrn.com/abstract=4481469}},
  author = {Nicoloso, Sara},
  title = "Vertex-colouring of some
$3$-chromatic circulant graphs
$C_n(a, b, c)$ with $\text{gcd}(n, a)\;\text{mod }3 = 0$",
  howpublished = "Preprint",
  year = "2023"
}

\end{document}